\documentclass[10pt]{amsart}
\usepackage{amssymb}
\usepackage[all]{xy}
\usepackage[toc,page,title,titletoc,header]{appendix}
\usepackage{amsmath,amssymb,epsfig,amscd,amsthm,verbatim}
\usepackage[bookmarks,colorlinks,breaklinks]{hyperref}
\usepackage{verbatim}

\begin{document}
\theoremstyle{plain}
\newtheorem{thm}{Theorem}[section]
\newtheorem*{thm1}{Theorem 1}
\newtheorem*{thm1.1}{Theorem 1.1}
\newtheorem*{thmM}{Main Theorem}
\newtheorem*{thmA}{Theorem A}
\newtheorem*{thm2}{Theorem 2}
\newtheorem{lemma}[thm]{Lemma}
\newtheorem{lem}[thm]{Lemma}
\newtheorem{cor}[thm]{Corollary}
\newtheorem{pro}[thm]{Proposition}
\newtheorem{propose}[thm]{Proposition}
\newtheorem{variant}[thm]{Variant}
\theoremstyle{definition}
\newtheorem{notations}[thm]{Notations}
\newtheorem{rem}[thm]{Remark}
\newtheorem{rmk}[thm]{Remark}
\newtheorem{rmks}[thm]{Remarks}
\newtheorem{defi}[thm]{Definition}
\newtheorem{exe}[thm]{Example}
\newtheorem{claim}[thm]{Claim}
\newtheorem{ass}[thm]{Assumption}
\newtheorem{prodefi}[thm]{Proposition-Definition}
\newtheorem{que}[thm]{Question}
\newtheorem{con}[thm]{Conjecture}

\newtheorem*{smlprob}{Problem Skolem-Mahler-Lech}
\newtheorem*{pvprob}{Problem Picard-Vessiot}
\newtheorem*{tvcon}{Tate-Voloch Conjecture}
\newtheorem*{dmmcon}{Dynamical Manin-Mumford Conjecture}
\newtheorem*{dmlcon}{Dynamical Mordell-Lang Conjecture}
\newtheorem*{condml}{Dynamical Mordell-Lang Conjecture}
\numberwithin{equation}{section}
\newcounter{elno}                
\def\points{\list
{\hss\llap{\upshape{(\roman{elno})}}}{\usecounter{elno}}}
\let\endpoints=\endlist

\newcommand{\lra}{{\longrightarrow}}
\newcommand{\dra}{{\dashrightarrow}}
\newcommand{\Gr}{{\rm Gr}}
\newcommand{\GO}{{\rm GO}}
\newcommand{\Fan}{{(\F^{\an})}}
\newcommand{\Phian}{{(\Phi^{\an})}}
\newcommand{\lcm}{{\rm lcm}}
\newcommand{\Tor}{{\rm Tor}}
\newcommand{\perf}{{\rm perf}}
\newcommand{\ad}{{\rm ad}}
\newcommand{\Spa}{{\rm Spa}}
\newcommand{\Perf}{{\rm Perf}}
\newcommand{\alHom}{{\rm alHom}}
\newcommand{\SH}{\rm SH}
\newcommand{\Tan}{\rm Tan}
\newcommand{\res}{{\rm res}}
\newcommand{\Om}{\Omega}
\newcommand{\om}{\omega}
\newcommand{\OO}{\mathcal{O}}
\newcommand{\la}{\lambda}
\newcommand{\mc}{\mathcal}
\newcommand{\mb}{\mathbb}
\newcommand{\surj}{\twoheadrightarrow}
\newcommand{\inj}{\hookrightarrow}
\newcommand{\zar}{{\rm zar}}
\newcommand{\Exc}{\rm Exc}
\newcommand{\an}{{\rm an}}
\newcommand{\red}{{\rm \mathbf{red}}}
\newcommand{\codim}{{\rm codim}}
\newcommand{\Supp}{{\rm Supp}}
\newcommand{\rank}{{\rm rank}}
\newcommand{\Ker}{{\rm Ker \ }}
\newcommand{\Pic}{{\rm Pic}}
\newcommand{\Div}{{\rm Div}}
\newcommand{\Hom}{{\rm Hom}}
\newcommand{\im}{{\rm im}}
\newcommand{\Spec}{{\rm Spec \,}}
\newcommand{\Nef}{{\rm Nef \,}}
\newcommand{\Frac}{{\rm Frac \,}}
\newcommand{\Sing}{{\rm Sing}}
\newcommand{\sing}{{\rm sing}}
\newcommand{\reg}{{\rm reg}}
\newcommand{\Char}{{\rm char}}
\newcommand{\Tr}{{\rm Tr}}
\newcommand{\ord}{{\rm ord}}
\newcommand{\id}{{\rm id}}
\newcommand{\NE}{{\rm NE}}
\newcommand{\Gal}{{\rm Gal}}
\newcommand{\Min}{{\rm Min \ }}
\newcommand{\Max}{{\rm Max \ }}
\newcommand{\Alb}{{\rm Alb}\,}
\newcommand{\GL}{{\rm GL}}        
\newcommand{\PGL}{{\rm PGL}\,}
\newcommand{\Bir}{{\rm Bir}}
\newcommand{\Aut}{{\rm Aut}}
\newcommand{\End}{{\rm End}}
\newcommand{\Per}{{\rm Per}\,}
\newcommand{\ie}{{\it i.e.\/},\ }
\newcommand{\niso}{\not\cong}
\newcommand{\nin}{\not\in}
\newcommand{\soplus}[1]{\stackrel{#1}{\oplus}}
\newcommand{\by}[1]{\stackrel{#1}{\rightarrow}}
\newcommand{\longby}[1]{\stackrel{#1}{\longrightarrow}}
\newcommand{\vlongby}[1]{\stackrel{#1}{\mbox{\large{$\longrightarrow$}}}}
\newcommand{\ldownarrow}{\mbox{\Large{\Large{$\downarrow$}}}}
\newcommand{\lsearrow}{\mbox{\Large{$\searrow$}}}
\renewcommand{\d}{\stackrel{\mbox{\scriptsize{$\bullet$}}}{}}
\newcommand{\dlog}{{\rm dlog}\,}    
\newcommand{\longto}{\longrightarrow}
\newcommand{\vlongto}{\mbox{{\Large{$\longto$}}}}
\newcommand{\limdir}[1]{{\displaystyle{\mathop{\rm lim}_{\buildrel\longrightarrow\over{#1}}}}\,}
\newcommand{\liminv}[1]{{\displaystyle{\mathop{\rm lim}_{\buildrel\longleftarrow\over{#1}}}}\,}
\newcommand{\norm}[1]{\mbox{$\parallel{#1}\parallel$}}
\newcommand{\boxtensor}{{\Box\kern-9.03pt\raise1.42pt\hbox{$\times$}}}
\newcommand{\into}{\hookrightarrow}
\newcommand{\image}{{\rm image}\,}
\newcommand{\Lie}{{\rm Lie}\,}      
\newcommand{\CM}{\rm CM}
\newcommand{\sext}{\mbox{${\mathcal E}xt\,$}}  
\newcommand{\shom}{\mbox{${\mathcal H}om\,$}}  
\newcommand{\coker}{{\rm coker}\,}  
\newcommand{\sm}{{\rm sm}}
\newcommand{\pgcd}{\text{pgcd}}
\newcommand{\trd}{\text{tr.d.}}
\newcommand{\tensor}{\otimes}
\renewcommand{\iff}{\mbox{ $\Longleftrightarrow$ }}
\newcommand{\supp}{{\rm supp}\,}
\newcommand{\ext}[1]{\stackrel{#1}{\wedge}}
\newcommand{\onto}{\mbox{$\,\>>>\hspace{-.5cm}\to\hspace{.15cm}$}}
\newcommand{\propsubset}
{\mbox{$\textstyle{
\subseteq_{\kern-5pt\raise-1pt\hbox{\mbox{\tiny{$/$}}}}}$}}
\newcommand{\sA}{{\mathcal A}}
\newcommand{\sB}{{\mathcal B}}
\newcommand{\sC}{{\mathcal C}}
\newcommand{\sD}{{\mathcal D}}
\newcommand{\sE}{{\mathcal E}}
\newcommand{\sF}{{\mathcal F}}
\newcommand{\sG}{{\mathcal G}}
\newcommand{\sH}{{\mathcal H}}
\newcommand{\sI}{{\mathcal I}}
\newcommand{\sJ}{{\mathcal J}}
\newcommand{\sK}{{\mathcal K}}
\newcommand{\sL}{{\mathcal L}}
\newcommand{\sM}{{\mathcal M}}
\newcommand{\sN}{{\mathcal N}}
\newcommand{\sO}{{\mathcal O}}
\newcommand{\sP}{{\mathcal P}}
\newcommand{\sQ}{{\mathcal Q}}
\newcommand{\sR}{{\mathcal R}}
\newcommand{\sS}{{\mathcal S}}
\newcommand{\sT}{{\mathcal T}}
\newcommand{\sU}{{\mathcal U}}
\newcommand{\sV}{{\mathcal V}}
\newcommand{\sW}{{\mathcal W}}
\newcommand{\sX}{{\mathcal X}}
\newcommand{\sY}{{\mathcal Y}}
\newcommand{\sZ}{{\mathcal Z}}
\newcommand{\A}{{\mathbb A}}
\newcommand{\B}{{\mathbb B}}
\newcommand{\C}{{\mathbb C}}
\newcommand{\D}{{\mathbb D}}
\newcommand{\E}{{\mathbb E}}
\newcommand{\F}{{\mathbb F}}
\newcommand{\G}{{\mathbb G}}
\newcommand{\HH}{{\mathbb H}}
\newcommand{\I}{{\mathbb I}}
\newcommand{\J}{{\mathbb J}}
\newcommand{\M}{{\mathbb M}}
\newcommand{\N}{{\mathbb N}}
\renewcommand{\P}{{\mathbb P}}
\newcommand{\Q}{{\mathbb Q}}
\newcommand{\R}{{\mathbb R}}
\newcommand{\T}{{\mathbb T}}
\newcommand{\U}{{\mathbb U}}
\newcommand{\V}{{\mathbb V}}
\newcommand{\W}{{\mathbb W}}
\newcommand{\X}{{\mathbb X}}
\newcommand{\Y}{{\mathbb Y}}
\newcommand{\Z}{{\mathbb Z}}

\newcommand{\fix}{\mathrm{Fix}}

\title{Algebraic dynamics of skew-linear self-maps} 

\author{Dragos Ghioca}
\address{
Dragos Ghioca\\
Department of Mathematics\\
University of British Columbia\\
Vancouver, BC V6T 1Z2\\
Canada
}
\email{dghioca@math.ubc.ca}

\author{Junyi Xie}
\address{Junyi Xie, IRMAR, Campus de Beaulieu,
b\^atiments 22-23
263 avenue du G\'en\'eral Leclerc, CS 74205
35042  RENNES C\'edex}
\email{junyi.xie@univ-rennes1.fr}


\begin{abstract}
Let $X$ be a variety defined over an algebraically closed field $k$ of characteristic $0$, let $N\in\N$, let $g:X\dra X$ be a dominant rational self-map, and let $A:\A^N\lra \A^N$ be a linear transformation defined over $k(X)$, i.e., for a Zariski open dense subset $U\subset X$, we have that for $x\in U(k)$, the specialization $A(x)$ is an $N$-by-$N$ matrix with entries in $k$. We let $f:X\times \A^N\dra X\times \A^N$ be the rational endomorphism given by $(x,y)\mapsto (g(x), A(x)y)$. We prove that if the determinant of $A$ is nonzero and if there exists $x\in X(k)$ such that its orbit $\OO_g(x)$ is Zariski dense in $X$, then either there exists a point $z\in (X\times \A^N)(k)$ such that its orbit $\OO_f(z)$ is Zariski dense in $X\times \A^N$ or there exists a nonconstant rational function $\psi\in k(X\times \A^N)$ such that $\psi\circ f=\psi$. Our result provides additional evidence to a conjecture of Medvedev and Scanlon. 
\end{abstract}

\thanks{The first author is partially supported by a Discovery Grant from the National Sciences and Engineering Research Council of Canada, while the second author is partially supported by project ``Fatou'' ANR-17-CE40-0002-01.}

\maketitle

\section{Introduction}

\subsection{Notation}

We let $\N_0:=\N\cup\{0\}$.  
Throughout our paper, we let $k$ be an algebraically closed field of characteristic $0$. Also, unless otherwise noted, all our subvarieties are assumed to be closed. In general, for a set $S$ contained in an algebraic variety $X$, we denote by $\overline{S}$ its Zariski closure.

For a variety $X$ defined over $k$ and endowed with a rational self-map $\Phi$, for any subvariety $V\subseteq X$, we define $\Phi(V)$ be the Zariski closure of the set $\Phi\left(V\setminus I(\Phi)\right)$, where $I(\Phi)$ is the indeterminacy locus of $\Phi$; in other words, $\Phi(V)$ is the strict transform of $V$ under $\Phi$. Also, we denote by $\OO_\Phi(\alpha)$ the orbit of any point $\alpha\in X(K)$ under $\Phi$, i.e., the set of all $\Phi^n(\alpha)$ for $n\in \N_0$ (as always in algebraic dynamics, we denote by $\Phi^n$ the $n$-th compositional power of the map $\Phi$, where $\Phi^0$ is the identity map, by convention). We say that $\alpha$ is periodic if there exists $n\in\N$ such that $\Phi^n(\alpha)=\alpha$; furthermore, the smallest positive integer $n$ such that $\Phi^n(\alpha)=\alpha$ will be called the \emph{period} of $\alpha$. We say that $\alpha$ is preperiodic if there exists $m\in\N_0$ such that $\Phi^m(\alpha)$ is periodic. More generally, for an irreducible subvariety $V\subset X$, we say that $V$ is periodic if $\Phi^n(V)= V$ for some $n\in\N$; if $\Phi(V)=V$ (i.e., $\overline{\Phi\left(V\setminus I(\Phi)\right)}=V$), we say that $V$ is invariant under the action of $\Phi$ (or simpler, invariant by $\Phi$).

We will also encounter the following setup in our paper. Given a variety $X$ defined over $k$ and given $N\in\N$, we consider some $N$-by-$N$ matrix $A$ whose entries are rational functions on $X$; when the determinant of $A$ is  nonzero, then we write $A\in \GL_N(k(X))$. For any $N$-by-$N$ matrix $A\in M_{N,N}(k(X))$ there exists an open, Zariski dense subset $U\subset X$ such that for each $x\in U$, the matrix $A(x)$ obtained by evaluating each entry of $A$ at $x$ is well-defined. We call \emph{skew-linear self-map} a rational self-map $f:X\times \mathbb{A}^N\dashrightarrow X\times \mathbb{A}^N$ of the form $f(x, y)=(g(x), A(x)y)$, where $g:X\dashrightarrow X$ is a given rational self-map, while $A\in M_{N,N}(k(X))$.


\subsection{Zariski dense orbits}


The following conjecture was proposed by Medvedev and Scanlon \cite[Conjecture 5.10]{Medvdevv1} (and independently, by Amerik and Campana \cite{Amerik2008}); see also Zhang's \cite[Conjecture~4.1.6]{Zhang} regarding Zariski dense orbits 
for polarizable endomorphisms which motivated the aforementioned conjecture.

\begin{con}\label{conexistszdo}
Let $X$ be a quasiprojective variety over $k$ and $f:X\dashrightarrow X$ be a dominant rational self-map for which there exists no nonconstant
rational function $\psi\in k(X)$ such that $\psi\circ f=\psi$.
Then there exists a point $x\in X(k)$ whose orbit is Zariski dense in $X$.
\end{con}

The condition from Conjecture~\ref{conexistszdo} that there is no nonconstant rational function $\psi\in k(X)$ such that $\psi\circ f = \psi$ is also refered as saying that $f$ does not fix a nonconstant fibration. It is immediate to see that such a condition is absolutely necessary in order to hope for the conclusion in Conjecture~\ref{conexistszdo} to hold; the difficulty in Conjecture~\ref{conexistszdo} is to prove that such a condition is indeed  sufficient for the existence of a Zariski dense orbit when the ground field is countable (note that the case when $k$ is uncountable was established first in \cite{Amerik2008}).  

In order to state our results, we introduce first the following definition.
\begin{defi}
Let $X$ be any projective variety over $k$ and $f:X\dashrightarrow X$ be a dominant rational self-map. We say that the pair $(X,f)$ is \emph{good} if Conjecture~\ref{conexistszdo} holds for every pair which is birationally equivalent to $(X,f)$, i.e., Conjecture~\ref{conexistszdo} holds for any dynamical system $(Y,g)$ for which there exists a birational map $\psi:X\dra Y$ such that $\psi\circ f=g\circ \psi$. 
\end{defi}

\begin{rem}
\label{rem:birationally equivalent}
It is immediate to see that if $(X,f)$ and $(Y,g)$ are birationally equivalent, then $f$ fixes a nonconstant fibration if and only if $g$ fixes a nonconstant fibration. Furthermore, if there is a point with a Zariski dense orbit under $f$ in each (nontrivial) open subset of $X$, then for any pair $(Y,g)$, which is birationally equivalent to $(X,f)$, there exists a point in $Y$ with a Zariski dense orbit under $g$. 
\end{rem}

Conjecture~\ref{conexistszdo} predicts that each dynamical pair $(X,f)$ is good; furthermore, in each of the important instances when Conjecture~\ref{conexistszdo} holds for $(X,f)$, then we actually know that the pair $(X,f)$ is good (for more details, see Section~\ref{subsection Zariski dense results}).   

We prove the following result for skew-linear self-maps. 

\begin{thm}\label{thmzardenserat}
Let $g:X\dashrightarrow X$ be a dominant rational map defined over $k$, let $N\in\N$, and let $f:X\times\A^N_k\dashrightarrow X\times\A^N_k$ be a dominant rational map defined by $(x,y)\mapsto (g(x),A(x)y)$ where  $A\in \GL_{N}(k(X))$. If the pair $(X,g)$ is good, then the pair $(X\times\A^N_k,f)$ is good.
\end{thm}

In Section~\ref{subsection Zariski dense results} we discuss various cases when Conjecture~\ref{conexistszdo} is known to hold; our Theorem~\ref{thmzardenserat} provides extensions of each one of those results since in the cases when Conjecture~\ref{conexistszdo} is known to hold for a dynamical pair $(X,f)$, then actually $(X,f)$ is a good pair.

Very importantly, we note that the study of the dynamics of pairs $(X\times \A_k^N,f)$ where $f(x,y)=(g(x),A(x)y)$ for some endomorphism $g:X\lra X$ and some $A\in \GL_N(k(X))$ is quite subtle. Even in the special case when  $X=\G_m^\ell$,  $g:\G_m^\ell\lra \G_m^\ell$ is an algebraic group endomorphism and $A\in \G_N(k)$ is a constant matrix, it is a delicate question to get a complete characterization for which $g$, $A$ and $x\in (\G_m^\ell\times \A^N)(k)$ we have that $\OO_f(x)$ is Zariski dense. This last question is completely solved in \cite{MS-linear.algebraic} using purely diophantine tools, thus very different techniques from the ones employed in our present paper.


\subsection{A brief history of previous results for the conjecture on the existence of Zariski dense orbits}
\label{subsection Zariski dense results}


We work with the notation as in Conjecture~\ref{conexistszdo}.

The special case of Conjecture~\ref{conexistszdo} when $k$ is an uncountable field was proved in \cite[Theorem~4.1]{Amerik2008} (which is stated more general, in the setting of K\"{a}hler manifolds); also, when $k$ is uncountable, but in the special case $f$ is an automorphism, Conjecture~\ref{conexistszdo} was independently proven in \cite[Theorem~1.2]{Dixmier}. Furthermore, if $k$ is uncountable, Conjecture~\ref{conexistszdo} holds even when $k$ has positive characteristic (see \cite[Corollary 6.1]{dynamical-Rosenlicht}). If $k$ is  countable, Conjecture~\ref{conexistszdo} has only been proved in 
a few special cases, using various techniques ranging from number theory, to $p$-adic dynamics, to higher dimensional algebraic geometry.

First, we note that Conjecture~\ref{conexistszdo} holds if $X$ has strictly positive Kodaira dimension and $f$ is birational, as proven in \cite[Theorem~1.2]{BGRS}. 

For varieties of negative Kodaira dimension, we note that Medvedev and Scanlon \cite[Theorem~7.16]{Medvdevv1} proved
Conjecture~\ref{conexistszdo}  for endomorphisms $f$ of $X=\mathbb{A}^m$ 
of the form $f(x_1,\dots, x_m)=(f_1(x_1),\dots, f_m(x_m))$, 
where $f_1,\dots, f_m\in k[x]$. Combining techniques from model theory, 
number theory and polynomial decomposition theory, they obtain a complete description of all invariant subvarieties, which is the key 
 to Conjecture~\ref{conexistszdo} since orbit closures are invariant.

In the case when $X$ is an abelian variety and $f \colon X \to X$ 
is a dominant self-map,  Conjecture~\ref{conexistszdo} was proved 
in~\cite{MS-ab-var}. The proof uses the explicit description
of endomorphisms of an abelian variety and relies on 
the Mordell-Lang conjecture, due to  Faltings~\cite{Faltings}. The strategy from \cite{MS-ab-var} was then extended in \cite{MS-semiabelian} to prove Conjecture~\ref{conexistszdo} for all regular self-maps of any semiabelian variety.

Using methods from valuation theory (among several other tools), the second author proved in  \cite[Theorem~1.1]{Xie} another important special case of  Conjecture~\ref{conexistszdo} for all polynomial endomorphisms $f$ of $\mathbb{A}^2$. Previously, the same author established in \cite{Xie-Duke} the validity of Conjecture~\ref{conexistszdo} for all birational automorphisms of surfaces (see also \cite{BGT-IMRN} for an independent proof in the case of automorphisms of surfaces).  

Finally, we observe that Conjecture~\ref{conexistszdo} may be viewed as a dynamical analogue of a theorem of Rosenlicht (see \cite{dynamical-Rosenlicht} for a comprehensive discussion on this theme). More precisely, the following result was proven by Rosenlicht \cite[Theorem~2]{rosenlicht}.

\begin{thm} 
\label{thmzardenalgg}  {\rm (\cite[Theorem~2]{rosenlicht})}
Consider the action of an algebraic group $G$ on 
an irreducible algebraic variety $X$ defined over an algebraically closed field $k$ of characteristic $0$. There exists a $G$-invariant dense open subvariety
$X_0 \subset X$ and a $G$-equivariant morphism
$g : X_0 \lra Z$ (where $G$ acts trivially on $Z$), 
with the following properties:
\begin{enumerate}
\item[(i)] for each $x \in X_0(k)$,
the orbit $G \cdot x$ equals the fiber $g^{-1}(g(x))$; and 
\item[(ii)] $g^* k(Z) = k(X)^G:=\left\{\psi\in k(X)\colon \psi\circ h=\psi\text{ for each }h\in G\right\}$.
\end{enumerate}
In particular, if there is no nonconstant fibration fixed by $G$, then for each $x\in X_0(k)$, we have $G\cdot x=X_0$ is Zariski dense in $X$. 
\end{thm}

Theorem~\ref{thmzardenalgg} yields that Conjecture~\ref{conexistszdo} holds for each automorphism $f:X\lra X$ contained in an algebraic group $G$ (acting on $X$). Indeed (see also \cite{dynamical-Rosenlicht}) one can apply Theorem~\ref{thmzardenalgg} to $X$ and the algebraic group $G_0$ which is the Zariski closure of the cyclic group spanned by $f$ inside  $G$ and thus get that if $f$ does not fix a nonconstant fibration, then there is $x\in X(k)$ such that $G_0\cdot x$ is dense in $X$, and therefore $\OO_f(x)$ is Zariski dense in $X$ as well.


\subsection{Invariant subvarieties}


As a by-product of our method, we obtain the following characterization of invariant subvarieties under skew-linear automorphisms of $\A^1\times \A^N$ of the form $(x,y)\mapsto (x+1, A(x)y)$, where $A\in \GL_{N}(k[x])$.

\begin{thm}
\label{thmautoinvar}
Let $f:\A^1_k\times\A^N_k\to \A^1_k\times\A^N_k$ be an automorphism defined by $(x,y)\mapsto (x+1,A(x)y)$ where  $A(x)$ is a matrix in $\GL_{N}(k[x])$. Then there exists an automorphism $h$ on $\A_k^1\times \A^N_k$ of the form $(x,y)\mapsto (x,T(x)y)$ where $T(x)\in \GL_N(k[x])$ such that  for each  subvariety $V$ (not necessarily irreducible) of $\A^1_k\times\A^N_k$ invariant under $f$, we have 
$h^{-1}(V)=\A^1_k\times V_0\subseteq \A^1_k\times \A^N_k$ where $V_0$ is a subvariety of $\A^{N}_k$. 
\end{thm}

We also prove in Theorem~\ref{thmrateninvar} a more general version of the above result for invariant subvarieties under the action of a skew-linear self-map $f:X\times \A^N\dra X\times \A^N$. 

\begin{rem}
\label{remark straight model}
With the notation as in Theorem~\ref{thmautoinvar}, we have that $h^{-1}(V)$ is invariant under $h^{-1}\circ f\circ h$;  in other words, $h^{-1}\circ f\circ h=(x+1, B(x)y)$ and $B(x)(V_0)=V_0$ for all $x\in \A^1_k.$
\end{rem}

A skew-linear automorphism $\tilde{f}:\A^1\times \A^N\lra\A^1\times\A^N$ such as the automorphism $h^{-1}\circ f\circ h$ from Remark~\ref{remark straight model} will be called \emph{straight}; more precisely, an automorphism of $\A^1\times \A^N$ of the form $(x,y)\mapsto (x+1,A(x)y)$ is straight if each invariant subvariety under its action is of the form $\A^1\times V_0$ for some subvariety $V_0\subseteq \A^N$ (see also Definition~\ref{definition straight model}). Theorem~\ref{thmautoinvar} yields that any automorphism $f$ of $\A^1\times \A^N$ of the form $(x,y)\mapsto (x+1, A(x)y)$ is conjugate to a straight automorphism (see Remark~\ref{remark straight model}). In Section~\ref{section straight model} we study more in-depth the straight automorphisms of $\A^1\times \A^N$, which leads us to proving the following result.

\begin{thm}
\label{corperiod 0}
Let $N\in\N$, let $A\in \GL_N(k[x])$, let $f:\A^1\times \A^N\lra \A^1\times \A^N$ be the automorphism given by $(x,y)\mapsto (x+1,A(x)y)$, and let $V$ be a periodic subvariety of $\A^1\times \A^N$ under the action of $f$. Then the period of $V$ is uniformly bounded by a constant depending only on $A$ (and independent of $V$). 
\end{thm}

Actually, in Corollary~\ref{corperiod} we prove a more precise version of Theorem~\ref{corperiod 0} by showing that the period of any periodic subvariety $V$ divides some positive integer intrinsically associated to $A$. 
We believe that Theorem~\ref{corperiod 0} (and more generally, the results from Section~\ref{section automorphisms of affine space}) would be helpful in a further study of finding which points $x\in \A^1_k\times \A^N_k$ have a Zariski dense orbit under an automorphism $f$ of the form $(x,y)\mapsto (x+1,A(x)y)$. 

Besides the intrinsic interest in the results of Section~\ref{section automorphisms of affine space}, they also provide a simpler proof of a special case of Theorem~\ref{thmrateninvar}, thus helping the reader to understand the more general approach from Section~\ref{section rational maps}.


\subsection{The plan for our paper}


In Section~\ref{section rational maps} we study the invariant subvarieties for skew-linear self-maps $f$ of $X\times \A^N$ (for an arbitrary algebraic variety $X$) and subsequentely prove Theorems~\ref{thmrateninvar}~and~\ref{thmzardenserat}. In Section~\ref{section automorphisms of affine space} we prove Theorem~\ref{thmautoinvar} (which is a more precise version of Theorem~\ref{thmrateninvar} when $X=\A^1$ and $f$ is an automorphism) and then Theorem~\ref{corperiod 0} (see Corollary~\ref{corperiod}). We conclude our paper with a more in-depth study of straight forms corresponding to skew-linear automorphisms of $\A^1\times\A^2$; see Section~\ref{subsection N=2}.

\smallskip

{\bf Acknowledgments.} We are grateful to our colleagues Philippe Gille, Matthieu Romagny and Zinovy Reichstein for helpful conversations while preparing our paper. We also thank the anonymous referee for his/her useful comments and suggestions.


\section{Zariski dense orbits}
\label{section rational maps}


In this section, we let $X$ be a variety defined over an algebraically closed field $k$ of characteristic $0$, endowed with a dominant self-map $g:X\dashrightarrow X$. We let $N\in\N$ and let $\pi: X\times \A^N_k\to X$ be the projection onto the first coordinate. We also let $A\in \GL_N(k(X))$ and (as in Theorem~\ref{thmzardenserat}), we let $f:X\times \A^N\dra X\times \A^N$ be the rational endomorphism given by $(x,y)\mapsto (g(x),A(x)y)$.


\subsection{Characterization of invariant subvarieties}


An important ingredient in our proof of Theorem~\ref{thmzardenserat} is a complete description of the subvarieties $Y$ of $X\times \A^N$, which dominate $X$ under the projection map $\pi$, and moreover, $Y$ is invariant under the action of the skew-linear self-map $f$. So, we start by stating Theorem~\ref{thmrateninvar} which characterizes the (not necessarily irreducible) subvarieties of $X\times \A^N$, which are invariant under the rational self-map $f$; we state our result  under the assumption that $g$ fixes no nonconstant rational fibration, i.e., there is no nonconstant $\phi\in k(X)$ such that $\phi\circ g=\phi$. In Section~\ref{subsection preserves fibration} we explain that the general case can be reduced to Theorem~\ref{thmrateninvar}. 

\begin{thm}\label{thmrateninvar}Let $f:X\times\A^N_k\dashrightarrow X\times\A^N_k$ be a dominant rational map defined by $(x,y)\mapsto (g(x),A(x)y)$ where  $A(x)$ is a matrix in $\GL_{N}(k(X))$, and let $\pi:X\times \A^N\lra X$ be the projection map. Suppose that there is no nonconstant rational function $\phi\in k(X)$ such that $\phi\circ g=\phi.$ Then there exists:
\begin{itemize}
\item an integer $\ell\geq 1$;
\item an irreducible variety $Y$ endowed with a dominant rational map $g':Y\dra Y$ along with a generically finite map $\tau: Y\dashrightarrow X$ satisfying $\tau\circ g'=g^\ell\circ \tau$; 
\item a birational map $h$ on $Y\times \A^N_k=Y\times_{X} X\times\A^N_k$ of the form $(x,y)\mapsto (x,T(x)y)$ where $T(x)\in \GL_N(k(Y))$, 
\end{itemize}
such that  for any (not necessarily irreducible) subvariety $V\subset X\times\A^N_k$ with the properties that:
\begin{itemize}
\item $V$ is invariant under $f$, and 
\item each irreducible component of $V$  dominates $X$ under the induced projection map $\pi|_{V}:V\lra X$, 
\end{itemize}
we have $h^{-1}\left((\tau\times_X\id)^{\#} (V)\right)=Y\times V_0\subseteq Y\times \A^N_k$, where $V_0$ is a subvariety of $\A^{N}_k$ and $(\tau\times_X\id)^{\#}(V)$ is the corresponding strict transform\footnote{Let $\phi:X_1\dashrightarrow X_2$ be any generically finite rational map between projective varieties.  
Let $W$ any subvariety of $X_2$, we define the strictly transform $\phi^{\#}(W)$  of $W$ to be the union of all irreducible components with the multiplicities of the Zariski closure of $\phi^{-1}|_{X_1\setminus I(\phi)}(W)$ on which $\phi$ are generically finite.}. 
\end{thm}

Theorem~\ref{thmrateninvar} is a generalization of Theorem~\ref{thmautoinvar} (though the latter result is slightly more precise, i.e., $\ell=1$ if $X=\A^1$ and $g(x)=x+1$).  
We prove Theorem~\ref{thmrateninvar} in Section~\ref{subsection proof theorem invariant subvarieties}.


\subsection{Invariant cycles}
Denote by  $t_g:=[k(X):g^*(k(X))]\geq 1$ the topological degree of $g$.

For any irreducible subvariety $W$ of $X\times \A^N_k$ which dominates $X$, denote by $f_{\#}W:=d_Wf(W)$ where  $d_W$ is the topological degree of $f|_{W}$ (and, as always, $f(W)$ is the Zariski closure of $f(W\setminus I(f))$). In our case, since $W$ dominates $X$ and the action of $f$ on the fiber is linear, we have $d_W=t_g$.

Let $V$ be an effective cycle of $X\times \A^N_k$ such that every irreducible component of $V$ dominates $X$.
Write $V=\sum_{i=1}^\ell a_iV_i$ where $V_i$ are irreducible components of $V$ and $a_i\geq 1$.
Write $f_\#V:=\sum_{i=1}^\ell a_if_\#V_i=t_g\sum_{i=1}^\ell a_if(V_i).$
We say that $V$ is invariant under $f$ if the support of $V$ and $f_\#(V)$ are the same i.e. $f_{\#} V=t_g V$.

For any subvariety $V$ of $X\times \A^N_k$ such that every irreducible component of $V$ dominates $X$, we may view it as an effective cycle such that every irreducible component of $V$ dominates $X$ and all nonzero coefficients are equal to one. Then it is invariant under $f$ if and only if as an effective cycle, it is invariant under $f$.

\subsection{Characterization of invariant subvarieties, general case}
\label{subsection preserves fibration}


In this section we explain that the case in which $g$ fixes a nonconstant fibration can be reduced to Theorem~\ref{thmrateninvar}. Indeed, first of all, we may suppose that $X$ is projective (since $g$ is a rational self-map).  Then let 
$$L=k(X)^g=\left\{\phi\in k(X)\colon \text{ }\phi\circ g=\phi\right\};$$
clearly, $L$ is a subfield of $k(X)$ containing $k$. Let $r$ be the transcendence degree of $L$ over $k$; so, $1\le r\leq \dim X$ since we assume that $g$ fixes a nonconstant fibration. 

Let $R$ be a finitely generated $k$-subalgebra of $L$ whose fraction field is $L$.  Let $B$ be an irreducible projective variety containing $\Spec R$ as a dense subset. 
The inclusion $R\hookrightarrow k(X)$ yields a dominant rational map $\psi: X\dashrightarrow B$.  
At the expense of replacing $X$ by some suitable birational model, we may assume  that $X$ is smooth and that the map $\psi$ is regular.  By Stein factorization, we may further assume that the generic fiber of $\psi$ is connected. By generic smoothness, we obtain that the generic fiber of $\psi$ is smooth and thus geometrically irreducible.

Let $\eta$ be the generic point of $B$.  Let $K$ be an algebraic closure of $L$. The geometric generic fiber of $\psi$ is denoted by $X_{\eta}$ over $K.$ Then $g$ induces a dominant rational self-map $g_{\eta}$ on $X_{\eta}$ and $f$ induces a dominant rational self-map on $X_{\eta}\times \A^N$.  Denote by $I$ the set of invariant subvarieties of $X\times \A^N$ such that each of their irreducible components dominate $X$ under the projection map $X\times\A^N\lra X$; we also let  $I_{\eta}$ be the set of invariant subvarieties of $X_{\eta}\times \A^N$ such that each of their irreducible components dominate $X_{\eta}$.
For every invariant subvariety $V\in I$, we have that $V_{\eta}:=V\times_XX_{\eta}$ is contained in $I_{\eta}$; the map $V\mapsto V_{\eta}$ is bijective.

By the construction of $B$, there is no nonconstant rational function $\phi\in K(X_{\eta})$ satisfying $\phi\circ g=\phi$; therefore Theorem \ref{thmrateninvar} applies for $X_{\eta}$.  
So, there exists an integer $\ell\geq 1$ and an irreducible variety $Y_{\eta}$ endowed with a dominant rational self-map $g'_{\eta}:Y_{\eta}\to Y_{\eta}$ along with a generically finite map $\tau_{\eta}: Y_{\eta}\dashrightarrow X_{\eta}$ satisfying $\tau_{\eta}\circ g'_{\eta}=g_{\eta}^\ell\circ \tau_{\eta}$ such that there exists 
a birational map $h_{\eta}$ on $Y_{\eta}\times \A^N=Y_{\eta}\times_{X_{\eta}} X_{\eta}\times\A^N$ of the form $(x,y)\mapsto (x,T(x)y)$ where $T(x)\in \GL_N(K(Y_{\eta}))$ with the property that for any subvariety $V_{\eta}\in I_{\eta}$, we have 
$$h_{\eta}^{-1}((\tau_{\eta}\times_{X_{\eta}}\id)^{\#} (V_{\eta}))=Y_{\eta}\times V_0'\subseteq Y_{\eta}\times \A^N,$$ 
where $V_0'$ is a subvariety of $\A^{N}_K$ and $(\tau_{\eta}\times_{X_{\eta}}\id)^{\#}(V_{\eta})$ is the strict transform.  We note that $X_{\eta}$ is in fact defined over $L$; furthermore, there exists a finite extension $J$ over $L$ such that $Y_{\eta},\tau_{\eta}, h_{\eta}$ and $V_0'$ are defined over $J$.


\subsection{Proof of Theorem~\ref{thmrateninvar}}
\label{subsection proof theorem invariant subvarieties}


We work with the notation as in Theorem~\ref{thmrateninvar}.

\smallskip

Let $\sB$ be the set of points $x\in X$ such that $f$ is not a locally isomorphism on the fiber $\pi^{-1}(x)$. Then $\sB$ is a proper closed subset of $X$.

Let $I$ be the set of all effective invariant cycles $V$ in $X\times \A^N_k$ for which every irreducible component of $V$ dominates $X$ under the projection map $\pi:X\times \A^N\lra X$. 
For any $x\in X$ and for  any $V\in I$, we let 
$$V_x:=\pi^{-1}(x)\cap V\subseteq \A^{N}_k.$$
In the next result we show that over a Zariski dense subset of $X$, we have that each $V_x$ is obtained through some linear transformation from a given $V_{x_0}$. 
\begin{pro}\label{lemisofiberopen}
Let $V\in I$. 
There exists a Zariski open set  $U_V$ of $X$ such that  such that for any points $x_1, x_2\in U_V(k)$, there exists $g\in\GL_N(k)$ such that $V_{x_2}=g(V_{x_1}).$
\end{pro}

\begin{proof}
After replacing $X$ by some Zariski dense open subset, we may assume that there exists $d\geq 1$ such that $\deg V_x=d$ for all $x\in X(k).$
Let $M_d$ be the variety parametrizing all effective cycles in $\A^N_k$ of degree $d.$ Then $f$ induces a rational map 
$$F:X\times M_d\dashrightarrow X\times M_d\text{ given by }(x,W)\mapsto (g(x), A(x)(W)).$$ 
Let $\pi_1:X\times M_d\to X$ be the projection onto the first coordinate.

At the expense of replacing $X$ by some Zariski dense open subset, we may assume that the map $s$ given by  $x\mapsto (x, V_x)$ is a section from $X$ to $X\times M_d$. For any point $W\in M_d(k)$, there is a morphism 
$$\chi_{W}: X\times \GL_N\to X\times M_d\text{ given by }(x,g)\to (x, g(W)).$$  
We note that $\chi_{W}(X\times \GL_N)=\chi_{g(W)}(X\times \GL_N)$ for any $g\in \GL_N(k).$

The next Lemma yields (essentially) the conclusion in Proposition~\ref{lemisofiberopen}.
\begin{lem}\label{lemexvkiso}
There exists $W\in M_d(k)$ and a Zariski dense open set $U$ of $X$ such that $s(U)\subseteq \chi_{W}(X\times \GL_N)$.
\end{lem}

\begin{proof}[Proof of Lemma~\ref{lemexvkiso}.]
Let $K$ be an algebraically closed uncountable field containing $k.$
By \cite{Amerik2008} (see also \cite{dynamical-Rosenlicht}),  there exists a $K$-point $\alpha\in X(K)$ such that $g^n(\alpha)\not\in \sB$ for all $n\geq 0$ and its orbit is Zariski dense in $X_K.$

For all $n\in \N_0$, we have $s(g^n(\alpha))\in \chi_{V_\alpha}(X_K\times \GL_N)$. Hence $s^{-1}\left(\chi_{V_\alpha}\right)(X_K\times \GL_N)$ is a Zariski dense constructible set in $X_K$ and so,  it contains a Zariski dense open set $U_K$ in $X_K$.  Since $X(k)$ is Zariski dense in $X_K$, there exists a point $\beta\in U_K\cap X(k)$. It follows that there exists $g\in \GL_N(k)$ such that $g(V_\alpha)=V_\beta$. Then we have $V_\beta\in M_d(k)$ and $\chi_{V_\alpha}(X_K\times \GL_N)=\chi_{V_\beta}(X_K\times \GL_N)$.  Since $s$ and $\chi_{V_\beta}(X_K\times \GL_N)$
are both defined over $k$, then $s^{-1}\left(\chi_{V_\beta}\right)(X_K\times \GL_N)$ is a Zariski dense constructible set in $X_K$ which is defined over $k$. It follows that as a $k$-variety, $s^{-1}\left(\chi_{V_\beta}\right)(X\times \GL_N)$ is a Zariski dense constructible set in $X$. Then there exists a Zariski dense open set $U$ of $X$ such that $s(U)\subseteq \chi_{V_\beta}(X\times \GL_N)$, which concludes the proof of Lemma~\ref{lemexvkiso}.
\end{proof}

Now, let $U$ be as in the conclusion of Lemma~\ref{lemexvkiso}. 
Then for any $x_1,x_2\in U$ there are  $g_1,g_2\in \GL_N(k)$ such that $V_{x_1}=g_1(W)$ and $V_{x_2}=g_2(W)$. Therefore $V_{x_2}=g_2g_1^{-1}(V_{x_1})$, as desired in the conclusion of Proposition~\ref{lemisofiberopen}. 
\end{proof}

We observe that Proposition~\ref{lemisofiberopen} applies to each $V\in I$ and so,  we let $U_V$ be the Zariski open subset of $X$ satisfying the conclusion of Proposition~\ref{lemisofiberopen} with respect to the variety $V$. 
   
For any $V\in I$ and any points $\alpha, \beta\in U_V$, denote by $G^V_{\beta,\alpha}$ the set of $g\in \GL_N(k)$ such that $g(V_\beta)=V_\alpha$.   By Proposition \ref{lemisofiberopen}, the set $G^V_{\beta,\alpha}$ is nonempty, so let $g^V_{\beta,\alpha}$ be an element of $G^V_{\beta,\alpha}$. Then  $G^V_{\beta,\alpha}=g^V_{\beta,\alpha}G^V_{\beta,\beta}$; we note that $G^V_{\beta,\beta}$ is an algebraic subgroup of $\GL_N(k).$

The next result yields that the (a priori disjoint) sets $G_{\beta,\alpha}^V$ all contain some given sets $G_{\beta,\alpha}^S$ for a suitable invariant cycle $S\in I$.    
\begin{lem}\label{lemsdeterv}
There exists an effective invariant cycle $S\in I$ such that for any $V\in I$, there exists a Zariski dense open set $U\subseteq U_S\cap U_V$ with the property  that for any two points $x_1,x_2\in U(k)$ we have $G_{x_1,x_2}^S\subseteq G_{x_1,x_2}^V.$
\end{lem}

\begin{proof}

We note that, if $V_1,\dots, V_s$ are invariant effective cycles in $I$, then $\sum_{i=1}^sn_iV_i$ (for arbitrary $n_i\in \N$) is also contained in $I$.

Let $K$ be an algebraically closed field containing $k$ such that the cardinality of $K$ is strictly larger that the cardinality of $I$. 
Then there exists a point $\beta\in X(K)$ such that 
\begin{equation}
\label{point beta}
\beta\in \bigcap_{V\in I}U_{V}(K).
\end{equation}
For any $V\in I$, denote by $V_\beta:=V\cap \pi^{-1}(\{\beta\})$ the fiber of $V_K$ at the point $\beta\in X(K)$ from \eqref{point beta}. Let $$G_\beta^V:=\{g'\in \GL_N\colon g'(V_\beta)=V_\beta\};$$ 
then $G_\beta^V$ is an algebraic subgroup of $\GL_N$.
We also let 
$$G_\beta:=\bigcap_{V\in I}G^V_\beta;$$ 
then  there exists a finite subset $\{V_1,\dots,V_s\}\subseteq I$ such that $$G_\beta:=\bigcap_{i=1}^sG^{V_i}_\beta.$$ 
Let $M$ be the maximum of the multiplicities of all irreducible components of $(V_1)_\beta,\dots,(V_s)_\beta$ and let 
$$S:=\sum_{i=1}^s(M+1)^{i-1}V_i\in I.$$ 
Then for any $g'\in \GL_N(K)$, we have $g'(S_\beta)=S_\beta$ if and only if $g'((V_i)_\beta)=(V_i)_\beta$ for all $i=1,\dots,s.$ In other words, $$G^{S}_\beta=\bigcap_{i=1}^sG^{V_i}_\beta=G_\beta.$$
For any $V\in I$, denote by $A_V$ the maximum of all multiplicities of all irreducible components of $V$. Now for any $V\in I$, let $$M_V:=\max\{A_V,A_S\}+1\text{ and }W:=S+M_V\cdot V\in I,$$ 
and also let $U:=U_{W}\cap U_V\cap U_S$  
where the open sets $U_W$, $U_V$ and $U_S$ satisfy the conclusion of Proposition~\ref{lemisofiberopen}. For any $x_1,x_2\in U(k)$, we claim that 
\begin{equation}
\label{inclusion G x_1 x_2}
G_{x_1,x_2}^S\subseteq G_{x_1,x_2}^V. 
\end{equation}
Since both $G^S_{x_1,x_2}$ and $G^V_{x_1,x_2}$ are defined over $k$ and $k$ is algebraically closed, we only need to show the inclusion \eqref{inclusion G x_1 x_2} after base change $K/k$. So, we only need to show that $G_{x_1,x_2}^S(K)\subseteq G_{x_1,x_2}^V(K)$.
Since $\beta\in  U_{W}(K)$, for any $i=1,2$, there exists $g_{\beta,x_i}$ satisfying 
$$g_{\beta,x_i}\left(S_{\beta}+M_V\cdot V_\beta\right)=S_{x_i}+M_V\cdot V_{x_i}.$$ 
It follows that $g_{\beta,x_i}(S_{\beta})=S_{x_i}$ and $g_{\beta,x_i}(V_\beta)=V_{x_i}.$  Then we have $$G_{x_1,x_2}^S(K)=g_{\beta,x_2}G_\beta^S(K)g_{\beta,x_1}^{-1}= g_{\beta,x_2}G_\beta(K)g_{\beta,x_1}^{-1}$$ and $$G_{x_1,x_2}^V(K)=g_{\beta,x_2}G_\beta^V(K)g_{\beta,x_1}^{-1}.$$
Since $G_\beta\subseteq G^V_\beta$, we have $G_{x_1,x_2}^S(K)\subseteq G_{x_1,x_2}^V(K)$, as desired in Lemma~\ref{lemsdeterv}.
\end{proof}

Now we have all ingredients necessary to finish the proof of Theorem~\ref{thmrateninvar}. 

\begin{proof}[Proof of Theorem~\ref{thmrateninvar}.]
Fix a point $\alpha\in U_S(k)$. Then $G_\alpha:=G_{\alpha,\alpha}^S$ is an algebraic subgroup of $\GL_N.$
Let $\sG$ be the subvariety of $U_S\times \GL_N$ of points $(x,g')\in U_S\times \GL_N$ such that $S_x=g'(S_\alpha)$.  Lemma~\ref{lemisofiberopen} yields that $\sG$ is a $G_\alpha$-torsor on $U_S.$ Denote by 
$$p:\sG\subseteq U_S\times \GL_N\to U_S$$ 
the projection on the first coordinate. For any $x\in U_S,$ let $G_x:=G_{\alpha,x}^S$.
We note that for any $x_1,x_2\in U_S$, we have $G^S_{x_1,x_2}=G^S_{x_2}G_{x_1}^{-1}$. Note that for any $x\in (U_S\setminus \sB)\cap g|_{U_S\setminus \sB}^{-1}(U_S)$, we have $g'\in G_x$ and $A(x)g'(S_\alpha)=A(x)S_x=S_{g(x)}.$
Then $f$ induces a dominant rational map $F$ on $\sG$ defined by  
$(x,g')\mapsto (g(x), A(x)g')$.

Let $G^0_\alpha$ be the connected component of $G_\alpha$; also let $\mu:= G_\alpha/G^0_\alpha$, which is a finite group. Then the quotient $Y':=\sG/G^0_\alpha$ is a $\mu$-torsor on $U_S.$ Observe that $F$ induces a rational self-map $f'$ on $Y'$ such that $\pi'\circ f'=g\circ \pi'$ where $\pi': Y'\to U_s$ is the projection to the base $U_S.$ 
Let $Y$ be an irreducible component of $Y'$. Then there exists $\ell\geq 1$ such that $f^{'\ell}(Y)=Y.$ Let $g':=f^{'\ell}|_Y$ and $\tau:= \pi'|_Y$. Then we have $\tau\circ g'=g^\ell\circ \tau$.

Now we consider the dominant rational map $f_Y:=\id\times_X f$ on the base change $Y\times_X X\times \A^N$. Let $\sG_Y:=Y\times_{U_S}\sG$. Then $\sG_Y/G^0_\alpha=Y\times_{U_S}Y'$ has a section 
$$T_0:Y\to Y\times_{U_S}Y\subseteq \sG_Y/G^0_\alpha\text{ sending }y\to (y,y).$$   The preimage $\sG_Y^0$ of $T(Y)$ in $\sG_Y$ is a connected component of $\sG_Y$ which is a $G_\alpha^0$ torsor on $Y.$ By \cite{Colliot-Thelene1992},
there exists a rational section 
$$T: Y\to \sG_Y^0\text{ satisfying }p_Y\circ T=\id,$$ 
where $p_Y$ is the projection from $\sG_Y$ to $Y$ and $T(\alpha)=1 \in G_\alpha$ (i.e., $T(\alpha)$ is the identity element of $G_\alpha$). We note that for any $x\in Y$, we have $T(x)\in G_{\tau(x)}.$

Let $h$ be the rational map on $Y\times \A^N_k$ defined by $(x,y)\mapsto (x,T(x)y).$ Let $V\in I$ be an invariant subvariety of $X\times \A^N_k.$
For any point $x\in Y$, denote by $(\tau\times_X\id)^{-1} (V)_x$ the fiber of $(\tau\times_X\id)^{-1} (V)$ at $x$. As a subvariety in $A^N_k$, we have $(\tau\times_X\id)^{-1} (V)_x=V_{\tau(x)}.$

By Lemma \ref{lemsdeterv}, there exists a Zariski dense open set $U\subseteq U_S\cap U_V$ such that for any two points $x_1,x_2\in U(k)$ we have $G_{x_1,x_2}^S\subseteq G_{x_1,x_2}^V.$

Pick a point $u_1\in \tau^{-1}U(k)$. Let $V_0:=T(u_1)^{-1}((\tau\times_X\id)^{-1} (V)_{u_1})$. For any $u_2\in \tau^{-1}U(k)$, let $x_1=\tau(u_1)$ and $x_2=\tau(u_2)$.
Since $T(u_i)\in G_{x_i}$ for $i=1,2$, we have $T(u_2)T(u_1)^{-1}\in G_{x_1,x_2}^S\subseteq G_{x_1,x_2}^V.$ It follows that $T(u_2)T(u_1)^{-1}(V_{x_1})=V_{x_2}$. We have 
$$V_0=T(u_1)^{-1}((\tau\times_X\id)^{-1} (V)_{u_1})=T(u_1)^{-1}(V_{x_1})$$ 
$$=T(u_2)^{-1}(T(u_2)T(u_1)^{-1}(V_{x_1}))
=T(u_2)^{-1}(V_{x_2})=T(u_2)^{-1}((\tau\times_X\id)^{-1} (V)_{u_2}).$$
Then we get $h^{-1}(V)=Y\times V_0$, which concludes the proof of Theorem~\ref{thmrateninvar}.
\end{proof}


\subsection{Proof of Theorem~\ref{thmzardenserat}}
\label{subsection proof theorem dense orbits}


We work under the hypotheses of Theorem~\ref{thmzardenserat}. 

Let $\sB$ be the set of points $x\in X$ such that $f$ is not a locally isomorphism on the fiber $\pi^{-1}(x)$. Then $\sB$ is a proper closed subset of $X$.

If there exists a nonconstant rational function $\psi$ on $X$ invariant under $g$, then the nonconstant rational function $\psi\circ \pi$ on $X\times \A^N_k$ is invariant under $f$.
So Theorem~\ref{thmzardenserat} holds.
Now we may assume that there is no nonconstant rational function on $X$ invariant under $g$. Then there exists a Zariski dense orbit in $X(k)$ under the action of $g$.
Moreover, for any Zariski dense open set $U$ of $X$,  since the pair $(U,g|_U)$ is birationally equivalent to $(X,g)$, then there exists a point $x_U\in U(k)$  with a Zariski dense orbit under the action of $g|_U$.

\smallskip

Let $I$ be the set of all invariant subvarieties in $X\times \A^N_k$ for which every irreducible component of $V$ dominates $X$ under the projection map $X\times \A^N\lra X$.

Theorem~\ref{thmrateninvar} yields that (perhaps, at the expense of  replacing $f$ by a suitable iterate) 
 there exists an irreducible variety $Y$ endowed with a dominant rational self-map 
$$g':Y\dashrightarrow  Y$$ 
and a generically finite map $\tau: Y\dashrightarrow X$ satisfying $\tau\circ g'=g\circ \tau$ 
 such that there exists 
a birational map $h$ on $Y\times \A^N_k=Y\times_{X} X\times\A^N_k$ of the form  $(x,y)\mapsto (x,T(x)y)$ where $T(x)\in \GL_N(k(Y))$ such that  for any subvariety $V\in I$, we have 
$$h^{-1}((\tau\times_X\id)^{\#} (V))=Y\times V_0\subseteq Y\times \A^N_k,$$  where $V_0$ is a subvariety of $\A^{N}_k$.  Let $f':Y\times \A^N\to Y\times \A^N$ be the rational map defined by 
$$g'\times_{(X,g)} f:(x,y)\mapsto (g'(x), A(\tau(x))y).$$   
We have $(\tau\times \id)\circ f'=f\circ (\tau\times \id)$.
Let 
$$F:=h^{-1}\circ f'\circ h: Y\times \A^N\to Y\times \A^N.$$ 
Then $F$ is the map $(x,y)\mapsto (g'(x),B(x)y)$ where $B(x):=T^{-1}(g'(x))A(\tau(x))T(x)$.
Let $\rho:=(\tau\times \id)\circ h$. Then we have $\rho\circ F=f\circ \rho$. For any $V\in I$, we see that $\rho^{\#}(V)$ is invariant by $F$ and it has the form $Y\times V_0.$

After replacing $Y$ by some smaller open subset, we may assume that $\rho$ is a regular morphism.  Furthermore, we may assume that $\rho$ is locally finite.
Let 
$$p: Y\times \A^N\to Y$$ 
be the projection to the first coordinate. 
Let $\sB'$ be the set of points $x\in Y$ such that $F$ is not locally an   isomorphism on the fiber $p^{-1}(x)$. Then $\sB'$ is a proper closed subset of $Y$. 
There exists a point $\alpha\in X(k)$, such that $\OO_g(\alpha)\cap \sB=\emptyset$; here we use the assumption about $(X,g)$ being a good dynamical pair (so, in particular, there exists a point with a Zariski dense orbit contained in the complement of $\sB$). At the expense of replacing $\alpha$ by some $g^n(\alpha)$, we may suppose that there exists a point $\beta\in Y$ such that  $\tau(\beta)=\alpha$ and so, $\OO_{g'}(\beta)\cap \sB'=\emptyset$. Also, we may suppose that $T(\beta)=\id$.

\smallskip

For any $x\in X$ and $V\in I$, denote by $V_x:=\pi^{-1}(x)\cap V\subseteq \A^{N}_k$.
By Lemma \ref{lemisofiberopen},
there exists a Zariski open set  $U_V$ of $X$ such that  such that for any points $x_1, x_2\in U_V(k)$, there exists $g'\in\GL_N(k)$ such that $V_{x_2}=g'(V_{x_1}).$
There exists $m\geq 0$, such that $g^m(\alpha)\in U_V$.  There exists an open set $U'$ containing $\alpha$, such that $g^i(U')\cap \sB=\emptyset$ for $i=0,\dots,m$ and moreover, $g^m(U')\subseteq U_V.$
Then for any points $x_1, x_2\in U'(k)$, there exists $g'\in\GL_N(k)$ such that $$A(g^{m-1}(x_2))\cdots A(x_2)V_{x_2}=g'A(g^{m-1}(x_1))\cdots A(x_1)(V_{x_1});$$ it follows that 
$$V_{x_2}=\left(A(g^{m-1}(x_2))\cdots A(x_2)\right)^{-1}g'A(g^{m-1}(x_1))\cdots A(x_1)(V_{x_1}).$$ 
So we may replace $U_V$ by $U'$ and therefore assume that $\alpha\in U_V$ for all $V\in I.$

For any $V\in I$, any points $x_1,x_2$ in $U_V$, denote by $G^V_{x_1,x_2}$ the set of $g'\in \GL_N(k)$ such that $g'(V_{x_1})=V_{x_2}$. Then there exists an element $g^V_{x_1,x_2}\in G^V_{x_1,x_2}$. We note that $G^V_\alpha:=G^V_{\alpha,\alpha}$ is an algebraic subgroup of $\GL_N(k).$
Let $G_\alpha:=\cap_{V\in I}G^V_\alpha$; this is an algebraic subgroup of $\GL_N.$ 
For any $V\in I$, we have $\rho^{\#}(V)=\rho^{-1}(V)=Y\times V_\alpha$.  Hence   $B(x)\in G^V_\alpha$ for all $x\in Y\setminus \sB'$ and thus $B(x)\in G_\alpha$ for all $x\in Y\setminus \sB'.$

Theorem~\ref{thmzardenalgg} 
 shows that either there exists a point $y\in \A^N(k)$ such that $G_\alpha\cdot y$ is Zariski dense in $\A^N$ or there exists a nonconstant rational function $\phi\in k(\A^N)$ such that $\phi\circ g'=\phi$ for all $g'\in G_\alpha$.

At first, we suppose that there exists a point $y\in \A^N(k)$ such that $G_\alpha\cdot y$ is Zariski dense in $\A^N$. Furthermore, Theorem~\ref{thmzardenalgg} yields that \emph{each} point in a dense open subset of $\A^N$ would have a Zariski dense orbit under the action of $G_\alpha$. Now, let $\gamma:=(\alpha,y)\in X\times \A^N$. Denote by $Z$ the Zariski closure of $\OO_f(\gamma)$.  Since $\OO_g(\alpha)$ is Zariski dense in $X$, then $Z$ has at least one irreducible component which dominates $X$. Let $V$ be the union of all irreducible components of $Z$ which dominate $X$; then $V\in I$. There exists $m\geq 0$ such that  $f^m(\alpha)\in V$ and so, $f^n(\alpha)\in V$ for all $n\geq m.$

Let $\gamma '$ be the unique preimage of $\gamma$ under $\rho$ in the fiber $\pi^{-1}(\beta).$  Since we have assumed that $T(\alpha)=\id$, we have $\gamma '=(\beta,y).$ Then $$f^{'m}(\gamma ')\in \rho^{-1}(V)=\rho^{\#}(V)=Y\times V_\alpha.$$
It follows that $B\left(g^{'(m-1)}(\beta)\right)\cdots B(g'(\beta))\cdot B(\beta)y\in V_\alpha$. Since $$B\left(g^{'(m-1)}(\beta)\right)\cdots B(g'(\beta))\cdot B(\beta)\in G_\alpha\subseteq G_\alpha^V,$$ 
we have $y\in V_\alpha$. Then we have $G_\alpha\cdot y\subseteq V_\alpha$. Since $G_\alpha\cdot y$ is Zariski dense in $\A^N$, we have $V_\alpha=\A^N$. Then $\rho^{-1}(V)=Y\times A^N$. It follows that $V=X\times \A^N.$ So $\OO_f(\gamma)$ is Zariski dense in $X\times \A^N.$ 

Furthermore, we see that since any $\gamma=(\alpha,y)$ would have a Zariski dense orbit under $f$, where $\alpha$ is a point with a Zariski dense orbit under $g$ avoiding $\sB$ and therefore (since the pair $(X,g)$ is good), $\alpha$ may be chosen in any open subset of $X$, while $y$ is any point in a given open subset of $\A^N$, we have that there exist points with Zariski dense orbits under $f$ in any nontrivial, open subsets of $X\times \A^N$. Hence, for any other dynamical pair $(W,h)$, which is birationally equivalent to $(X\times \A^N,f)$, there exist $k$-points in $W$ with a Zariski dense orbit under $h$ (see Remark~\ref{rem:birationally equivalent}).

Now we assume that there exists a nonconstant rational function $\phi\in k(\A^N)$ such that $\phi\circ g'=\phi$ for all $g'\in G_\alpha$.
Let $\chi$ be the rational function on $Y\times A^N$ defined by $(x,y)\mapsto \phi(y)$; it is invariant by $f'$.  Let $\psi$ be the rational function on $X\times \A^N$ defined by 
$$\psi(x)=\prod_{x'\in \rho^{-1}(x)}\chi(x').$$ 
Then $\psi$ is a nonconstant rational function on $X\times \A^N$ invariant under $f$; according to Remark~\ref{rem:birationally equivalent}, each dynamical pair $(W,h)$, which is equivalent with $(X\times \A^N,f)$, also fixes some  nonconstant fibration. This concludes the proof of Theorem~\ref{thmzardenserat}.


\section{A special class of automorphisms of the affine space}
\label{section automorphisms of affine space}


In this section we study in-depth the special case in Theorem~\ref{thmrateninvar} when $X=\A^1$ and $f:\A^1\times \A^N\lra \A^1\times \A^N$ is an automorphism given by $(x,y)\mapsto (x+1, A(x)y)$ for some $A\in\GL_N(k[x])$. This leads to proving Theorem~\ref{thmautoinvar} and also to developing a theory of \emph{straight models} (see Subsection~\ref{section straight model}) for linear transformations $A\in \GL_N(k[x])$, which we believe is of independent interest. In particular, we believe our results would be helpful for understanding better which points in $\A^1_k\times \A^N_k$ have Zariski dense orbits under an automorphism $f$ as above.


\subsection{Proof of Theorem~\ref{thmautoinvar}}
\label{section invariant subvarieties}


We work under the hypotheses of Theorem~\ref{thmautoinvar}. So, $N$ is a positive integer,  $A\in \GL_N(k[x])$ and $f:\A^1\times \A^N\lra \A^1\times \A^N$ is an automorphism given by $(x,y)\mapsto (x+1, A(x)y)$.

For each $x\in \A^1(k)$, and each subvariety $V$ invariant under $f$, we let  $$V_x:=\pi^{-1}(x)\cap V\subseteq \A^{N}_k.$$ 
The next result is a more precise version of Proposition~\ref{lemisofiberopen} in our setting. 

\begin{lem}\label{lemisofiber}
For each $x\in \A^1(k)$, there exists $g_x\in\GL_N(k)$ such that $V_x=g_x(V_0).$
\end{lem}

\begin{proof} 
Let $d=\deg V_0$; then $\deg V_x=d$ for all $x\in \A^1(k)$. 
Let $M_d$ be the variety parametrizing all subvarieties of $\A^N_k$ of degree $d.$ Then $f$ induces an automorphism 
$$F:\A^1_k\times M_d\lra \A^1_k\times M_d\text{ defined by }(x,W)\mapsto (x+1, A(x)(W)).$$ 
Denote by $\pi_1:\A^1_k\times M_d\lra \A^1_k$ the projection to the first coordinate. There exists a section $s:\A^1_k\lra \A^1_k\times M_d$ defined by $x\mapsto (x, V_x)$ and there exists a morphism 
$$\chi: \A^1_k\times \GL_N\lra \A^1\times M_d\text{ given by }(x,g)\to (x, g(V_0)).$$    
For all $n\in \Z$, we have that $s(n)\in \chi(\A^1_k\times \GL_N)$; therefore   $s^{-1}\left(\chi(\A^1_k\times \GL_N)\right)$ is a Zariski dense constructible set in $\A^1_k$, thus it is a Zariski dense open subset of $\A^1_k$.

Observe that $s(\A^1_k)$ and $\chi(\A^1_k\times \GL_N)$ are invariant under $F$ and so, 
$$s(\A^1_k)\cap \chi(\A^1_k\times \GL_N)\text{ is also invariant under $F$.}$$
Thus $s^{-1}\left(\chi(\A^1_k\times \GL_N)\right)$ is invariant under $x\mapsto x+1$. Then $s^{-1}\left(\chi(\A^1_k\times \GL_N)\right)=\A^1_k.$  Therefore for any $x\in \A^1(k)$, there exists $g_x\in\GL_N(k)$ such that $V_x=g_x(V_0).$
\end{proof}

Let $\sG^V:=\{(x,g)\in \A^1_k\times \GL_N(k)\colon g(V_0)=V_x\}.$ Then $\sG^V$  is a subvariety of $\A^1_k\times \GL_N(k).$ Denote by $p_V:\sG^V\lra \A^1_k$ the projection onto the first coordinate.  For each $x\in \A^1(k)$, let  $G^V_x:=p_V^{-1}(x)$.  
We have $G_x^V=g_xG_0^V$. 

Let $I$ be the set of all invariant subvarieties in $\A^1_k\times \A^N_k$. Set $\sG:=\cap_{V\in I}\sG^V$; it is a subvariety of $\A^1_k\times \GL_N(k).$ 
Denote by $p:\sG\to \A^1_k$ the projection onto the first coordinate.  Set $G_x:=p^{-1}(x)$ for all $x\in \A^1(k).$ 
We have $G_x=g_xG_0$.  Then $\sG$ is a $G_0$-torsor on  $\A^1_k$; in the next result we will show that $\sG$ must be trivial.
 
\begin{lem}\label{lemtorsortrivialafl}
Any $G_0$-torsor $\sG$ on  $\A^1_k$ is trivial.
\end{lem}

\begin{proof} 
Let $G^0_0$ be the connected component of $G_0$, which is a normal subgroup of  $G_0$. Consider the exact sequence 
$$1\to G^0_0\to G_0\to G_0/G^0_0\to 1;$$ 
then we have the exact sequence 
$$H^1_{\acute{e}t}(\A^1_k, G^0_0)\to H^1_{\acute{e}t}(\A^1_k, G_0)\to H^1_{\acute{e}t}(\A^1_k, G_0/G^0_0).$$
Since $G_0/G^0_0$ is finite and $\A^1_k$ is simply connected,  then $H^1_{\acute{e}t}(\A^1_k, G_0/G^0_0)=1.$ So, we only need to show that $H^1_{\acute{e}t}(\A^1_k, G^0_0)=1.$

Let $R$ be the radical of $G^0_0$. Consider the exact sequence 
$$1\to R\to G^0_0\to G^0_0/R\to 1.$$
We get the exact sequence 
$$H^1_{\acute{e}t}(\A^1_k,R)\to H^1_{\acute{e}t}(\A^1_k, G^0_0)\to H^1_{\acute{e}t}(\A^1_k, G^0_0/R).$$
Since $G^0_0/R$ is semisimple, by \cite{Gille} (see also \cite{R-R})  we have $H^1_{\acute{e}t}(\A^1_k, G^0_0/R)=1.$
Thus we only need to show that $H^1_{\acute{e}t}(\A^1_k,R)=1.$
Since $R$ is solvable, all is left to prove is that $H^1_{\acute{e}t}(\A^1_k,\G_{m})$ and $H^1_{\acute{e}t}(\A^1_k,\G_{a})$ are trivial.
Obviously $H^1_{\acute{e}t}(\A^1_k,\G_{m})=\Pic(\A^1_k)$ is trivial and  $H^1_{\acute{e}t}(\A^1_k,\G_{a})=H^1(\A^1,O_{\A^1})=1$, by \cite[Theorem 3.5,~Chapter~III,~p.~215]{Hartshorne}. This  concludes our proof of Lemma~\ref{lemtorsortrivialafl}.
\end{proof}

So, there exists a section $T: \A^1\to \sG$ satisfying $p\circ T=\id$ and $T(0)=1\in G_0.$ Then $T\in \GL_N(k[x])$ and for all $x\in \A^1(k)$, we have  $T(x)\in g_xG_0.$
Let $h$ be the automorphism on $\A_k^1\times \A^N_k$ defined by $(x,y)\mapsto (x,T(x)y).$ Let $V\in I$ be an invariant subvariety of $\A^1_k\times \A^N_k$ under the action of $f$. 
Then for any $x\in \A^1(k)$, we have 
$$T(x)^{-1}(V_x)=T(x)^{-1}(g_x(V_0))=V_0,$$
and so, we have $h^{-1}(V)=\A^1_k\times V_0$, which concludes the proof of Theorem~\ref{thmautoinvar}.


\subsection{A straight model}
\label{section straight model}


In this section we continue our study of the dynamical properties of automorphisms $f$ of $\A^1\times \A^N$ of the form $(x,y)\mapsto (x+1, A(x)y)$. We  will prove Theorem~\ref{corperiod 0} (see Corollary~\ref{corperiod}) which says that each periodic subvariety of $\A^1\times \A^N$ under the action of $f$ has its period uniformly bounded depending only on the matrix $A$.

For any $A\in \GL_N(k[x])$,  denote by $f_A:\A^1_k\times\A^N_k \to \A^1_k\times\A^N_k$ the automorphism defined by $(x,y)\mapsto (x+1, A(x)y)$.

\begin{defi}
We say that $A$ and $A'$ are equivalent if $f_A$ and $f_{A'}$ are conjugate by an automorphism of $\A^1_k\times\A^N_k$ given by $(x,y)\mapsto (x,T(x)y)$, i.e., if there exists an element $T\in \GL_N(k[x])$, such that $A'(x)=T(x+1)^{-1}A(x)T(x).$ 
\end{defi}

Denote by $\sP^N$ the set of all subvarieties of $\A_k^N$ and denote by $\sP^N_1$ the set of all subvarieties of $\A^1_k\times \A_k^{N}$ which dominate $\A^1_k$ under the projection map $\A^1\times \A^N\lra \A^1$.  
We consider the map 
\begin{equation}
\label{definition r_0}
r_0:\sP^N_1\to \sP^N\text{ given by }V\mapsto V\cap \pi^{-1}(0)
\end{equation} 
and also, consider the section $\sigma: \sP^N\to \sP^N_1$ given by $W\mapsto \A^1_k\times W$; then we have $r_0\circ\sigma=\id.$

Let $I_A$ be the set of all subvarieties $V\in \sP^N_1$ which are invariant under $f_A.$ Let $I_A^0$ be the set of all subvarieties $W\in \sP^N$ such that $\sigma(W)$ is invariant under $f_A.$
We have $\sigma(I_A^0)\subseteq I_A$ and $I_A^0\subseteq r_0(I_A).$

\begin{lem}\label{lemrzeinjia}
The map $r_0|_{I_A}:I_A\lra \sP^N$ is injective.
\end{lem}

\begin{proof} Let $V_1,V_2$ be two elements in $I_A$. Then $V_1\cup V_2$ is also an element in $I_A$. 
 If $r_0(V_1)=r_0(V_2)$, then $r_0(V_1)=r_0(V_1\cup V_2).$   Lemma~\ref{lemisofiber} yields that 
$$\pi^{-1}(x)\cap V_1=\pi^{-1}(x)\cap (V_1\cup V_2)$$ 
for all $x\in \A^1(k).$ Then we have $V_1=V_2$, as desired. 
\end{proof}

\begin{defi}
\label{definition straight model}
We say that $A$ (or $f_A$) is \emph{straight}, if $r_0(I_A)=I_A^0$. 
\end{defi}

Lemma \ref{lemrzeinjia}, shows that $A$ is straight if and only if $I_A=\sigma(I_A^0)$, i.e., all invariant  subvarieties of $\A^1_k\times \A^N_k$ are of the form $\A^1_k\times W$.

For every $W\in \sP^N,$ denote by $G_W$ the subgroup  of $\GL_N$ which fixes $W.$ Let  
$$\G_A:=\bigcap_{W\in r_0(I_A)}G_W;$$ 
this is an algebraic subgroup of $\GL_N.$ Let $A'\in \G_A(k[x])$ be an element equivalent to $A$, i.e., $A'(x)=T^{-1}(x+1)A(x)T(x)$ where $T\in \GL_N(k[x])$. Then we have 
$$r_0(I_A')=T(0)^{-1}(r_0(I_A))\text{ and }\G_{A'}=T(0)^{-1}\G_AT(0).$$ 
So the conjugacy class of $\G_A$ in $\GL_N$ is an invariant in the equivalent class of $A$.

\begin{rem}
\label{remark straight 2}
Theorem \ref{thmautoinvar} (see also Remark~\ref{remark straight model}) yields  that for every $A\in \GL_N(k[x])$, there exists $A'\in \G_A(k[x])$ which is  straight and moreover, $A'$ and $A$ are equivalent.
\end{rem}

\begin{pro}\label{procristrga}
An element $A\in \GL_N(k[x])$ is straight if and only if $A(x)\in \G_A(k[x]).$
\end{pro}

\begin{proof}
First we suppose that $A$ is straight. For any $W\in I_A^0=r_0(I_A)$, we have that $\A^1_k\times W$ is invariant under $f_A$. It follows that $A(x)\in G_W(k[x])$. Then $$A(x)\in \bigcap_{W \in r_0(I_A)}G_W(k[x])=\G_A(k[x]).$$
If $A(x)\in \G_A(k[x])$, then for each $V\in I_A$, we have that $W:=r_0(V)$ is invariant under the action of $\G_A$. Then $\A^1_k\times W$ is invariant under the action of $f_A$. So, $\A^1_k\times W\in I_A$ and $r_0(\A^1_k\times W)\in I_A=V$. Therefore $V=\A^1_k\times W$, as claimed in the conclusion of Proposition~\ref{procristrga}.
\end{proof}

The next result yields a good criterion for when a point $\alpha\in \A^1(k)\times \A^N(k)$ has a Zariski dense orbit under $f$.

\begin{pro}\label{prostinva}Let $\alpha:=(a,b)\in \A^1(k)\times \A^N(k)$ and let $A\in \GL_N(k[x])$ be a straight linear transformation.  Then $\overline{\OO_{f_A}(\alpha)}=\A^1_k\times \overline{\G_A\cdot b}$. 
\end{pro}

\proof
Since $\A^1_k\times \overline{\G_A\cdot b}$ is $f_A$-invariant and $\alpha\in \A^1_k\times \overline{\G_A\cdot b}$, then $\overline{\OO_{f_A}(\alpha)}\subseteq \A^1_k\times \overline{\G_A\cdot b}$. Using that  $\overline{\OO_{f_A}(\alpha)}$ is $f_A$-invariant, we get that there exists $W\in r_0(I_A)$ such that $\overline{\OO_{f_A}(\alpha)}=\A^1_k\times W.$ By the definition of $\G_A$, we know that $W$ is $\G_A$-invariant. Since $b\in W$, we have $\overline{\G_A\cdot b}\subseteq W.$ Thus $\A^1_k\times \overline{\G_A\cdot b}\subseteq\overline{\OO_{f_A}(\alpha)}$, as desired.
\endproof

We are ready now to prove that each periodic subvariety under the action of $f_A$ has its period bounded depending only on $A$ (see Theorem~\ref{corperiod 0}).  
\begin{cor}\label{corperiod}
Let $V$ be a periodic subvariety of $f_A$ of period $m$. Then $m$ divides the number of connected components of $\G_A$. In particular, the period of each periodic subvariety of $\A^1\times \A^N$ under the action of $f_A$ is uniformly bounded by a constant depending only on $A$. 
\end{cor}
\begin{proof}
For each $i=0,\dots, m-1$, let $W_i:=r_0(f_A^i(V))$.  
We may assume that $A$ is straight (see Remark~\ref{remark straight 2}). Then $f^i(V)=\A^1_k\times W_i$. 
Since $\cup_{i=0}^{m-1}f_A^i(V)$ is invariant by $f_A$, then $\cup_{i=0}^{m-1}W_i$ is invariant by $\G_A$.  Therefore $\G_A$ acts on the set $\{W_0,\dots,W_{m-1}\}$ transitively, thus proving that $m$ must divide the number of components of $\G_A$. 
\end{proof}

\subsection{Straight forms when N is 2}
\label{subsection N=2}

In this section, let $f:\A^1\times \A^2\longrightarrow \A^1\times \A^2$ be an automorphism of the form $(x,y)\mapsto (x+1, A(x)y)$.

We say that an invariant subvariety $V$ of $f$ is nontrivial, if $V$ is not equal with $\A^1\times\{0\}$ or with $\A^1\times \A^2$.

\begin{lem}\label{lemdiacon}If $A=\left(\begin{array}{cc}a_1 & 0 \\0 & a_2\end{array}\right)$ where $a_1,a_2\in k^*$, then $f_A$ is straight.
\end{lem}

\proof
Let $V$ be a nontrivial invariant subvariety of $f$. We need to show that $V=\A^1\times r_0(V)$, where $r_0$ is defined in \eqref{definition r_0}. We argue by contradiction; also, we may assume that all irreducible components of $V$ have the same dimension. Thus  there are only two cases to consider: $\dim r_0(V)=0,1$.

At first, we assume that $\dim r_0(V)=1$. In this case, $V$ is defined by a polynomial $P(x,y_1,y_2)\in k[x,y_1,y_2]\setminus k[x]$. 
There exists $q\in k^*$ such that $f_A^*P=qP$ i.e.
$$P(x+1, a_1y_1,a_2y_2)=qP(x,y_1,y_2).$$
Write $P=\sum_I a_I(x)y^I$, where $I$ is the multi-index and $a_I(x)$ is a polynomial in $k[x]$. 
We get $\sum_I a_I(x+1)a^Iy^I=\sum_Iqa_I(x)y^I.$
Then we have $$a_I(x+1)=a^{-I}qa_I(x).$$ Comparing the coefficient of the leading term, we have $a^{-I}q=1$ if $a_I(x)\neq 0.$ Thus $a_I(x)\in k$ for any $I$ and so, $V=\A^1\times r_0(V)$.

Now we assume $\dim r_0(V)=0$.

Denote by $p_i: \A^1\times \A^2\to \A^1\times \A^1$ the projection  mapping $(x,y_1,y_2)$ to $(x,y_i)$ and let $f_i:  \A^1\times \A^1\to \A^1\times \A^1$ be the morphism $(x,y)\mapsto (x,a_iy).$ Then we have $p_i\circ f_A=f_i\circ p_i$ and $p_i(V)$ is an invariant subvariety of $f_i$ of codimension $1$ (for each $i=1,2$).

In this case, $p_i(V)$ is defined by a polynomial $P_i(x,y)\in k[x,y]\setminus k[x]$. 
There exists $q_i\in k^*$ such that $f_i^*P_i=q_iP_i$ i.e.
$$P_i(x+1, a_iy)=q_iP(x,y_i).$$
Write $P_i=\sum_j a_j(x)y^j$, where $a_i(x)$ is a polynomial in $k[x]$. 
We get $\sum_j a_j(x+1)a_i^jy^j=\sum_jq_ia_j(x)y^j.$
Then we have $$a_j(x+1)=a_i^{-j}q_ia_j(x).$$ Comparing the coefficient of the leading term, we get $a_i^{-j}q_i=1$ if $a_j(x)\neq 0.$ Then $a_j(x)\in k$ for any $j$ and so, $p_i(V)=\A^1\times r_0(p_i(V))$. Furthermore, since $\dim r_0(V)=0$, then also $\dim r_0(p_i(V))=0$.

Then we have $V\subseteq p_1^{-1}(p_1(V))\cap p_2^{-1}(p_2(V))=\A^1\times (r_0(p_1(V))\times r_0(p_2(V)))$.
We note that $r_0(p_1(V))\times r_0(p_2(V))$ is finite and $r_0(V)\subseteq r_0(p_1(V))\times r_0(p_2(V))$. So $V=\A^1\times r_0(V)$, as desired.
\endproof

\begin{pro}\label{proinvariimsp}Let $f$ be  an automorphism  of $\A^1\times \A^2$ of the form $(x,y)\mapsto (x+1, A(x)y)$. If there exists a nontrivial invariant subvariety of $f$, then there exists $B=\left(\begin{array}{cc}a_1 & 0 \\0 & a_2\end{array}\right)$ for some $a_1,a_2\in k^*$ such that $f$ is equivalent to $f_B$.
\end{pro}

\proof
Let $V$ be a nontrivial invariant subvariety of $f$.  We may assume that the dimension of all irreducible components of $V$ are the same.
By Theorem \ref{thmautoinvar}, we may suppose that $V=\A^1\times V_0$ where $V_0$ is a subvariety of $\A^2$ which is invariant under $A(x)$ for all $x\in k$. 

First, we observe that there exists $W_0:=\cup_{i=1}^sL_i\subseteq \A^2$ where $L_i$ are distinct lines passing through the origin such that $W=\A^1\times W_0$ is invariant under $f$.
If $\dim V_0=0$, we may take $W_0$ be the union of lines passing through the origin and a point in $V_0$ (other than the origin). 
If $\dim V_0=1$, we consider the standard embeding $\A^2\subseteq \P^2$ and then we may take $W_0$ be the union of lines passing through the origin and a point in the intersection of the Zariski closure of $V_0$ (in $\P^2$) and the line at infinity.

Now we may assume that $V$ takes form $V=\A^1\times \left(\bigcup_{i=1}^sL_i\right)=\bigcup_{i=1}^s\A^1\times L_i$. Moreover, we may assume that $f(\A^1\times L_i)=\A^1\times L_{i+1}$ for $i=1,\dots,s$ (where, by convention, we let $L_{s+1}:=L_1$).

\medskip

We have two cases: either $s=1$ or $s\ge 2$.

{\bf Case $s=1$.} In suitable coordinates, we may assume that $L_1$ is defined by $y_2=0$. Then with respect to these coordinates, we may further assume that 
$$A(x)=\left(\begin{array}{cc}a_1(x) & b(x) \\0 & a_2(x)\end{array}\right)$$ 
where $a_1(x),a_2(x), b(x)\in k[x].$ Because $\det A(x)=a_1(x)a_2(x)$ is a nonzero constant in $k[x],$ we have $a_1:=a_1(x)$ and $a_2:=a_2(x)$ are constants in $k^*$. We may assume that $b\neq 0$. Set $d:=\deg b(x)\geq 0.$ Denote by $k[x]_d$ the vector space of polynomials of degree at most $d$.

If $a_1\neq a_2$, consider the linear map $T: k[x]_d\to k[x]_d$ defined by $T:P(x)\mapsto a_2P(x+1)-a_1P(x)$. Next we analyze the leading term of $T(P)$; we have $\deg(T(P))=\deg(P)$. So $T$ is injective and therefore, it must be surjective as well. Hence there exists $u\in k[x]_d$, such that $T(u(x))=b(x).$
Let $U:\A^1\times \A^2\to \A^1\times \A^2$ be the automorphism of the form $$(x,y_1,y_2)\mapsto (x, y_1+u(x)y_2,y_2);$$ then $U^{-1}\circ f\circ U=f_B$ where 
$$B=\left(\begin{array}{cc}a_1 & 0 \\0 & a_2\end{array}\right).$$
\medskip

If $a:=a_1=a_2$, consider the linear map $T: k[x]_{d+1}\to k[x]_{d}$ defined by $T:P(x)\mapsto aP(x+1)-aP(x)$. If $v(x)\in \ker(L)$, we have $v(x+1)=v(x)$. Then $v(x)\in k$. It follows that $\ker T=k$. Since $\dim k[x]_{d+1}= \dim k[x]_{d}+1$, we obtain that $T$ is surjective. Hence there exists $u(x)\in k[x]_{d+1}$ such that $T(u)=b(x)$.
Let $U:\A^1\times \A^2\to \A^1\times \A^2$ be the automorphism given by $$(x,y_1,y_2)\mapsto (x, y_1+u(x)y_2,y_2);$$ 
we obtain that $T^{-1}\circ f\circ T=f_B$ where 
$$B=\left(\begin{array}{cc}a & 0 \\0 & a\end{array}\right).$$
\medskip

{\bf Case $s\geq 2$.} Then, in suitable coordinates, we may assume that $L_1$ is defined by $y_2=0$ and $L_2$ is defined by $y_1=0$. 
So, with respect to these coordinates, we may assume that 
$$A(x)=\left(\begin{array}{cc}0 & b(x) \\c(x) & d(x)\end{array}\right)$$ 
where $b(x),c(x), d(x)\in k[x].$ Because $\det A(x)=-b(x)c(x)$ is a nonzero constant in $k[x],$ we have $b:=b(x)$ and $c:=c(x)$ are constants in $k^*$. Then we have 
$$f(\A^1\times L_2)=\{(x, tb, td(x-1))\colon x,t\in k\}.$$
We note that $f(\A^1\times L_2)=\A^1\times L_3$. Therefore $d(x-1)$ must be a constant; so,  set $d:=d(x)\in k.$ Then we have $A(x)=\left(\begin{array}{cc}0 & b \\c & d\end{array}\right).$
Let $v$ be an eigenvector of $A$ in $k^2\setminus \{0\}$. Denote by $L$ the line in $\A^2$ spaned by $v$. Then $\A^1\times L$ is invariant by $f$. Then we reduced to the case $s=1$ and conclude our proof.
\endproof

Proposition \ref{proinvariimsp} implies the following result immediately.

\begin{cor}\label{corinvalin}If there exists a nontrivial invariant subvariety of $f$, then there exist an invariant trivial subbundle of rank 1 in the vector bundle $\A^1\times \A^2$ over $\A^1$. 

In other words,
there exist $a(x),b(x)\in k[x]$ satisfying $\gcd(a(x),b(x))=1$ and $c\in k^*$ such that $f(x,a(x),b(x))=(x+1,ca(x+1),cb(x+1))$.
\end{cor}

\proof
Assume that there exists a nontrivial invariant subvariety of $f$.
By Proposition \label{proinvariimsp}, we may assume that $f=f_B$ where $B=\left(\begin{array}{cc}a_1 & 0 \\0 & a_2\end{array}\right)$ for some $a_1,a_2\in k^*$. Then the subbundle of rank 1 defined by $y_1=0$ is invariant by $f$.

Now suppose that $L$ is an invariant subbundle of rank 1 in the vector bundle $\A^1\times \A^2$ over $\A^1$. Since $\Pic(\A^1)=\{0\}$, $L$ is trivial. There exists a everywhere nonzero section $s$ of $L$ i.e. there exist $a(x),b(x)\in k[x]$ satisfying $\gcd(a(x),b(x))=1$ and $c\in k^*$ such that $(x,a(x),b(x))\in L$ for all $x\in A^1.$ Since $L$ is invariant by $f$, the image of $f(s)$ under $f$ is also a nonzero section of $L$. Since $f(s)/s$ is a everywhere nonzero function on $\A^1$, it is constant. In other words $f(x,a(x),b(x))=(x+1,ca(x+1),cb(x+1))$ for some $c\in k^*$.
\endproof

We conclude by giving an example of an automorphism $f$ which has no nontrivial invariant subvariety.

\begin{pro}\label{proexanoin}If $A=\left(\begin{array}{cc}1 & 1 \\x & x+1\end{array}\right)$, then $f_A$ has no nontrivial invariant subvariety.
In particular, $f_A$ is straight.
\end{pro}
\proof
If $f_A$ has a nontrivial invariant subvariety, then 
by Corollary \ref{corinvalin},  there exist $a(x),b(x)\in k[x]$ satisfying $\gcd(a(x),b(x))=1$ and $c\in k^*$ such that $$f(x,a(x),b(x))=(x+1,ca(x+1),cb(x+1)).$$

It follows that $a(x)+b(x)=ca(x+1)$ and $xa(x)+(x+1)b(x)=cb(x+1).$ Then, combining these two equalities, we get:  $$b(x)=cb(x+1)-x(a(x)+b(x))=cb(x+1)-cxa(x+1).$$ 
We have then
$$\deg b(x)\geq \deg (cb(x+1)-b(x))=\deg (cxa(x+1))=1+ \deg a(x).$$
It follows that $$\deg b(x)=\deg (a(x)+b(x))=\deg(ca(x+1))=\deg a(x)\leq \deg b(x)-1.$$
Then we get a contradiction.
\endproof

\rem
Let $A=\left(\begin{array}{cc}1 & 1 \\x & x+1\end{array}\right)$. Proposition \ref{proexanoin} yields that $A$ is not equivalent to a constant matrix. 
\endrem



\begin{thebibliography}{BGKT99}


\bibitem[AC08]{Amerik2008}
E.~Amerik and F.~Campana.
\newblock {F}ibrations m\'{e}romorphes sur certaines vari\'{e}t\'{e}s \`{a} fibr\'{e} canonique trivial.  
\newblock {\em Pure Appl. Math. Q.} \textbf{4}(2, part 1):509--545, 2008.


\bibitem[BGR17]{dynamical-Rosenlicht} 
J.~P.~Bell, D.~Ghioca, and Z.~Reichstein.  
\newblock  On a dynamical version of a theorem of Rosenlicht. 
\newblock {\em Ann. Sc. Norm. Sup. Pisa. (5)} \textbf{17}(1):187--204, 2017.

\bibitem[BGRS17]{BGRS}
J.~P.~Bell, D.~Ghioca, Z.~Reichstein, and M.~Satriano.
\newblock On the Medvedev-Scanlon conjecture for minimal threefolds of non-negative Kodaira dimension.
\newblock {\em New York J. Math.} \textbf{23}:213--225, 2017.


\bibitem[BGT15]{BGT-IMRN} 
J.~P.~Bell, D.~Ghioca, and T.~J.~Tucker. 
\newblock Applications of p-adic analysis for bounding periods for subvarieties under \'etale maps. 
\newblock {\em Int. Math. Res. Not. IMRN} \textbf{2015}(11):3576--3597, 2015.




\bibitem[BRS10]{Dixmier} 
J.~P.~Bell, D.~Rogalski, and S.~J.~Sierra. 
\newblock The Dixmier-Moeglin equivalence for twisted homogeneous coordinate rings. 
\newblock {\em Israel J. Math.} \textbf{180}:461--507, 2010. 




\bibitem[CGP12]{Gille}
V.~Chernousov,  P.~Gille, and A.~Pianzola. 
\newblock {T}orsors over the punctured affine line.
\newblock {\em Amer. J. Math.} 134(6):1541--1583, 2012.


\bibitem[CO92]{Colliot-Thelene1992}
J.-L.~Colliot-Th\'{e}l\`{e}ne and M.~Ojanguren. 
\newblock {E}spaces principaux homog\`{e}nes localement triviaux. 
\newblock {\em Inst. Hautes \'{E}tudes Sci. Publ. Math.} \textbf{75}:97--122, 1992.




\bibitem[Fal94]{Faltings}
G.~Faltings.
\newblock The general case of {S}. {L}ang's conjecture.
\newblock In {\em Barsotti {S}ymposium in {A}lgebraic {G}eometry ({A}bano
  {T}erme, 1991)}, volume~15 of {\em Perspect. Math.}, pages 175--182. Academic
  Press, San Diego, CA, 1994.


\bibitem[GH]{MS-linear.algebraic}
D.~Ghioca and F.~Hu.
\newblock Density of orbits of endomorphisms of connected commutative linear algebraic groups.
\newblock submitted for publication, 12 pages, 2018.

\bibitem[GS]{MS-semiabelian}
D.~Ghioca and M.~Satriano.
\newblock Density of orbits of dominant regular self-maps of semiabelian varieties.
\newblock {\em Tran. Amer. Math. Soc.}, to appear, 18 pages, 2017.

\bibitem[GS17]{MS-ab-var}
D.~Ghioca and T.~Scanlon. 
\newblock Density of orbits of endomorphisms of abelian varieties.
\newblock {\em Tran. Amer. Math. Soc.} \textbf{361}(1):447--466, 2017. 






\bibitem[Har77]{Hartshorne}
R.~Hartshorne.
\newblock Algebraic geometry, volume \textbf{52} of {\em Graduate Texts in Mathematics}.
\newblock Springer-Verlag, New York-Heidelberg, 1977. xvi+496 pp. 






\bibitem[MS14]{Medvdevv1}
A. Medvedev and T. Scanlon.
\newblock Invariant varieties for polynomial dynamical systems.
\newblock {\em Ann. of Math. (2)}, \textbf{179}(1):81--177, 2014.













\bibitem[RR84]{R-R}
M.~S.~Raghunathan and A.~Ramanathan. 
\newblock {P}rincipal bundles on the affine line.
\newblock {\em Proc. Indian Acad. Sci. Math. Sci.} \textbf{93}(2-3):137--145, 1984.

\bibitem[Ros56]{rosenlicht}
M. Rosenlicht. 
\newblock Some basic theorems on algebraic groups. 
\newblock {\em Amer. J. Math.} {\bf 78}:401--443, 1956.















\bibitem[Xie15]{Xie-Duke} 
J.~Xie. 
\newblock Periodic points of birational transformations on projective surfaces.
\newblock {\em Duke Math. J.} \textbf{164}(5):903--932, 2015.


\bibitem[Xie]{Xie}
J.~Xie.
\newblock The existence of {Z}ariski dense orbits for polynomial endomorphisms of the affine plane.
\newblock {\em Compos. Math.}, to appear, 17 pages, 2015.







\bibitem[Zha06]{Zhang}
S.~Zhang. 
\newblock Distributions in algebraic dynamics.
\newblock In Surveys in differential geometry, volume~10 of Surv. Differ.
Geom., pages 381--430. Int. Press, Somerville, MA, 2006.

\end{thebibliography}

\end{document}